\documentclass[iccmp]{ipbook}
\usepackage{amsthm,amsmath,amssymb}
\usepackage{iftex}
\ifXeTeX
  \usepackage{xeCJK}
\else
  \usepackage{CJKutf8}
\fi
\usepackage{tikz}
\usepackage{booktabs}
\usepackage{tabularx}
\usepackage{array}
\usepackage{enumitem}
\usepackage{graphicx}
\usepackage{placeins}
\usepackage[mathlines]{lineno}
\usepackage{hyperref}
\usepackage{subcaption}
\hypersetup{colorlinks,linktocpage,linkcolor={blue},citecolor={blue}}
\newcolumntype{L}[1]{>{\raggedright\arraybackslash}p{#1}}

\startlocaldefs
\theoremstyle{plain}
\newtheorem{theorem}{Theorem}[section]
\newtheorem{lemma}{Lemma}[section]
\newtheorem{proposition}{Proposition}[section]

\theoremstyle{definition}
\newtheorem{definition}{Definition}[section]

\theoremstyle{remark}
\newtheorem{remark}{Remark}[section]

\newcommand{\diag}{\operatorname{diag}}
\newcommand{\tr}{\operatorname{tr}}

\newcommand{\argmin}{\operatorname*{arg\,min}}

\endlocaldefs
\newcommand{\R}{\mathbb{R}}
\newcommand{\U}{\mathbb{U}}

\newcommand{\KL}{{KL}}

\newcommand{\oslashc}{\oslash}

\newcommand{\expo}{\overset{\circ}{\operatorname{exp}}}
\firstpage{1}
\lastpage{1}

\title[Regularized Coupling Maps on Matrix Polytopes]{Regularized Coupling Maps on Matrix Polytopes:
From Assignment to Graph Matching and Optimal Transport}

\author{Binrui Shen}
\address{School of Mathematical Sciences, Laboratory of Mathematics and Complex Systems, MOE,
Beijing Normal University, Beijing 100875, P.R. China}
\address{Faculty of Arts and Sciences, Beijing Normal University, Zhuhai 519087, P.R. China}
\email{binrui.shen@bnu.edu.cn}

\author{Shengxin Zhu*}
\thanks{*Corresponding author}
\address{Research Centers for Mathematics, Advanced Institute of Natural Science,
Beijing Normal University, Zhuhai 519087, P.R. China}
\address{Guangdong Provincial Key Laboratory of Interdisciplinary Research and Application for Data Science,
BNU-HKBU United International College, Zhuhai 519087, P.R. China}
\email{Shengxin.Zhu@bnu.edu.cn}

\begin{document}

\begin{abstract}
\noindent
Three classical problems---linear assignment, graph matching, and optimal transport---share a common structure: they seek an optimal matrix-valued coupling on a matrix polytope. We study \emph{regularized coupling maps} that encompass two families of regularization: entropy/KL regularization, which exponentiates the input score matrix and projects it onto the relevant matrix polytope via the KL-projection; and the Frobenius/Euclidean regularization, which performs the Euclidean projection of the scaled input matrix onto the same polytope. We trace how these two geometries yield complementary algorithmic tools for graph matching (CSGO, ASM, FRAM) and optimal transport (IP-EOT, PSN). The regularized coupling map serves as a direction generator in graph matching and as a solution representation in optimal transport, with the regularizer determining the geometry and the outer task determining the precision standard.
\end{abstract}
\maketitle
\section{Introduction: a Unifying Perspective}

Three classical problems---\textbf{linear assignment}, \textbf{graph matching}, and \textbf{optimal transport}---all employ matrix-valued coupling variables. Yet the role of this coupling differs profoundly across these problems:
In \textbf{linear assignment}, the coupling is a permutation matrix encoding a one-to-one correspondence between two sets of equal cardinality.
In \textbf{graph matching}, the coupling records vertex correspondences and enters a quadratic objective, such as one of the Koopmans--Beckmann type.
In \textbf{optimal transport}, the coupling represents a transportation plan between two probability distributions with prescribed marginals.

Despite these differences, the feasible sets for the coupling variables in all three problems are naturally associated with \emph{matrix polytopes}: the Birkhoff polytope provides a convex relaxation for linear assignment and  graph matching, while the transportation polytope is the feasible set for discrete optimal transport. This observation suggests that much of the algorithmic and analytic machinery developed for one problem may illuminate the others.

The purpose of this exposition is to present a unified framework---{regularized coupling maps on matrix polytopes}---that encompasses several recent developments by the authors and their collaborators. We focus on two families of regularization: {entropy/KL regularization}, which leads to exponential kernels and the Sinkhorn matrix scaling, and {Frobenius/Euclidean regularization}, which corresponds to Euclidean projections onto the Birkhoff polytope. We show how these two geometries, though arising from different regularizers, yield complementary algorithmic tools for both graph matching and optimal transport.
This paper traces the flow from a generic regularized coupling map through two distinct geometric paths---entropy/KL and the Frobenius/Euclidean---and then to their respective applications in graph matching and optimal transport. 

\section{Matrix Polytopes and Regularized Assignment}

\subsection{The Birkhoff and Transportation Polytopes}

The {Birkhoff polytope} $\mathbb{U}$ is the set of $n \times n$ doubly stochastic matrices:
\begin{equation}
\mathbb{U} = \left\{ X \in \R_+^{n \times n} : X\mathbf{1}_n = \mathbf{1}_n,\, X^\top\mathbf{1}_n = \mathbf{1}_n \right\}.
\end{equation}
The celebrated Birkhoff-von Neumann theorem states that $\mathbb{U}$ is precisely the convex hull of the $n!$ permutation matrices; equivalently, the permutation matrices are the extreme points of $\mathbb{U}$ \cite{Birkhoff1946, vonNeumann1953}. This result is foundational: it tells us that relaxing the combinatorial assignment problem from permutation matrices to doubly stochastic matrices does not alter the optimal value, making the problem amenable to continuous optimization.

More generally, given marginal vectors $\mathbf{p} \in \R_+^m$ and $\mathbf{q} \in \R_+^n$ with $\mathbf{1}_m^\top\mathbf{p} = \mathbf{1}_n^\top\mathbf{q}$, the transportation polytope is defined as
\begin{equation}
\mathbb{U}^\mathbf{p}_\mathbf{q} = \left\{ X \in \R_+^{m \times n} : X\mathbf{1}_n = \mathbf{p},\, X^\top\mathbf{1}_m = \mathbf{q} \right\}.
\end{equation}
When $m=n$ and $\mathbf{p}=\mathbf{q}=\mathbf{1}_n$, the transportation polytope coincides with the Birkhoff polytope. Thus, the Birkhoff polytope is the special case.

\subsection{The Regularized Coupling Map}

Given a {score matrix} $S \in \R^{m \times n}$ and a convex regularizer $\Omega$, we define the {regularized coupling map} as follows. 

\begin{definition}[Regularized Coupling Map]
Let $\mathbb{P} \subset \R_+^{m \times n}$ be a matrix polytope (either the Birkhoff polytope $\mathbb{U}$ or a transportation polytope $\mathbb{U}^\mathbf{p}_\mathbf{q}$), let $S \in \R^{m \times n}$ be a score matrix, $\Omega: \mathbb{P} \to \R$ a continuous and strictly convex function, and $\tau > 0$ a regularization parameter. The {regularized coupling map} $S^{\tau,\Omega}_{\mathbb{P}}: \R^{m \times n} \to \mathbb{P}$ is defined by
\begin{equation}
S^{\tau,\Omega}_{\mathbb{P}} = \arg \max_{X \in \mathbb{P}} \ \langle X, S \rangle - \frac{1}{\tau} \Omega(X),
\end{equation}
where 
$\langle X, S \rangle = \sum_{i,j} X_{ij} S_{ij}$ is the Frobenius inner product.
\end{definition}

\noindent To simplify notation, we write $S^{\tau}_{\mathbb{P}}$ in place of $S^{\tau,\Omega}_{\mathbb{P}}$. As we shall see in \eqref{eq.entropy} and \eqref{eq.fro}, different $\tau$ correspond to different $\Omega$.

\begin{remark}[Well-definedness]
When $\Omega$ is strictly convex on $\mathbb{P}$, the objective $X \mapsto \langle X, S \rangle - \frac{1}{\tau} \Omega(X)$ is strictly concave on the convex compact set $\mathbb{P}$. Hence the maximizer exists and is unique, and the map is well-defined as a single-valued function. 
\end{remark}
In this exposition, we focus on two cases where strict convexity holds:
\begin{enumerate}
\item \textbf{Entropy regularization}: 
$\Omega(X)=\mathcal{H}(X)=\sum_{i,j} X_{ij} \log X_{ij}$, with the convention $0\log0=0$, is the \textit{discrete negative entropy} of a matrix $X \in \R_+^{m \times n}$.
The regularized objective is

\begin{equation}
    \langle X,S\rangle-\frac{1}{\beta}\mathcal{H}(X),
    \label{eq.entropy}
\end{equation}
with $\tau=\beta$ and $\beta>0$.

\item the \textbf{Frobenius regularization}: 
$\Omega(X)=\|X\|_F^2$, where $\|\cdot\|_F$ denotes the
Frobenius norm. The regularized objective is
\begin{equation}
\langle X,S\rangle-\frac{1}{\theta}\|X\|_F^2,
        \label{eq.fro}
\end{equation}
with $\tau=\theta$ and $\theta>0$.
\end{enumerate}
These two choices of $\Omega$ dominate the literature and will guide our exposition.
For general convex regularizers where strict convexity may fail, the definition should be understood as a set-valued map.

\begin{table}[h]
\centering
\caption{Birkhoff polytope on two regularization geometries. $\Gamma^{{KL}}_{\mathbb{U}}$ and $\Gamma^{{F}}_{\mathbb{U}}$: KL and Euclidean projections onto $\mathbb{U}$ (detailed in Sec. 3.1 and 4.1).}
\label{tab:comparison}

\renewcommand{\arraystretch}{1.2}
\begin{tabularx}{\linewidth}{
    @{}
    >{\raggedright\arraybackslash}p{0.20\linewidth}
    >{\raggedright\arraybackslash}X
    >{\raggedright\arraybackslash}X
    @{}
}
\toprule
\textbf{Aspect}
& \textbf{Entropy/KL Regularization}
& \textbf{Frobenius/Euclidean Regularization} \\
\midrule

Regularizer
& $\frac{1}{\beta}\mathcal{H}(X)$, where
  $\mathcal{H}(X)=\sum_{i,j}X_{ij}\log X_{ij}$
& $\frac{1}{\theta}\lVert X\rVert_F^2$, where
  $\lVert X\rVert_F^2=\sum_{i,j}X_{ij}^2$ \\
\addlinespace[0.6em]

Problem
& Entropy-Regularized Assignment (ERA)
& Frobenius-Regularized Assignment (FRA) \\
\addlinespace[0.6em]

Solution form
& $S^\beta_{\mathbb{U}}
  =\Gamma^{\KL}_{\mathbb{U}}\bigl(\expo(\beta S)\bigr)$
& $S_{\mathbb{U}}^\theta
  =\Gamma^{F}_{\mathbb{U}}\bigl(\theta S/2\bigr)$ \\
\addlinespace[0.6em]

Subsequent use
& Probabilistic directions in graph matching; entropic optimal transport
& Euclidean directions in graph matching \\

\bottomrule
\end{tabularx}
\end{table}

\section{The Entropy/KL Path: Softassign and Sinkhorn Scaling}

\subsection{Entropy-Regularized Assignment and Entropic Optimal Transport}
The {entropy-regularized assignment problem} (ERA) on the Birkhoff polytope is defined as the unique maximizer:
\begin{equation}
S^\beta_{\mathbb{U}} = \arg \max_{X \in \mathbb{U}}  \ \langle X, S \rangle - \frac{1}{\beta} \mathcal{H}(X) ,
\end{equation}
where $\beta > 0$ is the {inverse temperature} (sharpness parameter). This maximizer is unique and admits the explicit form of a KL projection \cite{Cuturi2013}:
\begin{equation}
S^\beta_{\mathbb{U}} = \Gamma^{\KL}_{\mathbb{U}} ( \expo(\beta S)) =D(\mathbf{u}) \expo(\beta S) D(\mathbf{v}),
\end{equation}
where $\Gamma^{\KL}_{\mathbb{U}}$ denotes the KL-projection operator onto $\mathbb{U}$. Specifically,  following \cite{Benamou2015}, for any positive kernel $K$, $\Gamma^{{KL}}_{\mathbb{U}}(K)$ is defined as the minimizer of the generalized Kullback--Leibler divergence $D_{{KL}}(\cdot\|K)$ over $\mathbb{U}$:
$$\Gamma^{{KL}}_{\mathbb{U}}(K) = \arg\min_{P \in \mathbb{U}} D_{{KL}}(P \| K),
\  D_{{KL}}(P\|K) = \sum_{i,j} ( P_{ij}\log\frac{P_{ij}}{K_{ij}} - P_{ij} + K_{ij} ).
$$
$D(\mathbf{u})$ and $D(\mathbf{v})$ are diagonal matrices constructed from positive scaling vectors $\mathbf{u}, \mathbf{v} \in \R_{++}^n$.

These vectors $\mathbf{u}, \mathbf{v}$ are computed via the {Sinkhorn matrix scaling}:
\begin{subequations}
\begin{align}
K_\beta(S)  &= \expo(\beta S),\\
\mathbf{u}^{(l+1)} &= \mathbf{1}_n \oslashc (K_\beta(S) \mathbf{v}^{(l)}),\\
\mathbf{v}^{(l+1)} &= \mathbf{1}_n \oslashc (K_\beta(S)^\top \mathbf{u}^{(l+1)}),
\end{align}
\end{subequations}
where $\oslashc$ and $\expo$ denote element-wise division and exponentiation, respectively. The {Sinkhorn--Knopp theorem} \cite{Sinkhorn1964, SinkhornKnopp1967} guarantees that for positive matrices, this alternating scaling converges to the unique doubly stochastic matrix in the positive diagonal equivalence class of $K_\beta(S)$.

The historical lineage of this algorithm is fascinating. The iterative proportional fitting procedure appeared in Deming and Stephan~\cite{Deming1940} for contingency table adjustment, in Kruithof~\cite{Kruithof1937} for telephone forecasting, and was rediscovered multiple times across economics, statistics, and transportation science \cite{Idel2016}. Sinkhorn's 1967 paper \cite{SinkhornKnopp1967} provided the first rigorous convergence proof for arbitrary positive matrices, establishing what is now known as the {Sinkhorn--Knopp algorithm} or {Sinkhorn matrix scaling}.

The connection between \textit{softassign} \cite{Gold1996} and optimal transport is immediate. The entropy-regularized optimal transport problem is:
\begin{equation}
S^\beta_{\mathbb{U}^\mathbf{p}_\mathbf{q}} = \arg \max_{X \in \mathbb{U}^\mathbf{p}_\mathbf{q}}  \langle X, S \rangle - \frac{1}{\beta} \mathcal{H}(X).
\label{eq.eot}
\end{equation}
\noindent  When $\mathbf{p} = \mathbf{q} = \mathbf{1}_n$, $S^\beta_{\U^{\mathbf{1}_n}_{\mathbf{1}_n}} = S^\beta_{\mathbb{U}}$. Thus, softassign can be considered as an entropy-regularized optimal transport with unit row and column marginals, and both share the same solution form:

\begin{equation}
   S^\beta_{\mathbb{U}^\mathbf{p}_\mathbf{q}} = \Gamma^{\KL}_{\mathbb{U}^\mathbf{p}_\mathbf{q}} ( \expo(\beta S)). 
\end{equation}

\subsection{Why Continuation Matters}
\label{sec:continuation}

In both graph matching and optimal transport, a larger $\beta$ yields a solution closer to the unregularized solution (Figure~\ref{fig:softassign-beta-effect} gives an example of this sharpening path), while the Sinkhorn iterations may become numerically unstable or slow to converge. Therefore, it is crucial to choose a moderate $\beta$ that balances accuracy and stability. A practical strategy is to solve a sequence of entropy-regularized problems with gradually increasing $\beta$ (i.e., decreasing temperature), and leverage the diminishing marginal effect to determine an appropriate parameter. The naive approach—restarting Sinkhorn matrix scaling from scratch for each new $\beta$—is computationally wasteful. The following theorem overcomes this inefficiency by enabling a ``warm start'' from a previously computed solution.
\begin{figure}[htbp]
\centering
\includegraphics[width=0.78\textwidth]{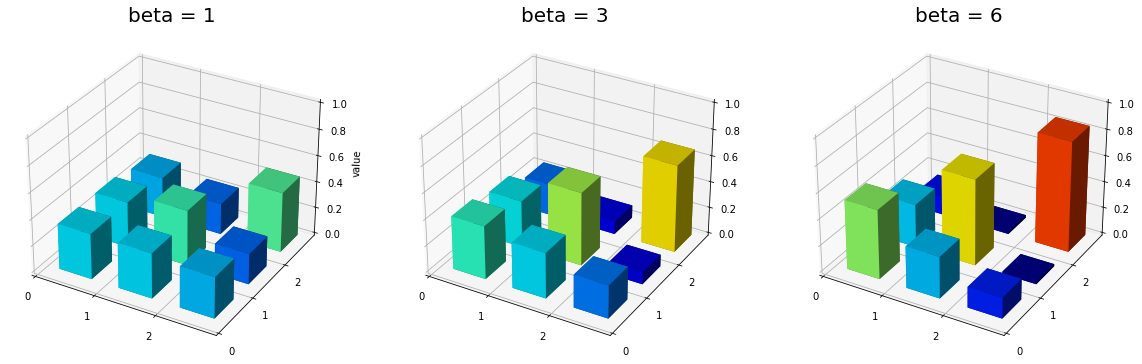}
\caption{Illustration of the sharpening effect of the inverse temperature $\beta$ in softassign \cite{Shen2024CVPR}.}
\label{fig:softassign-beta-effect}
\end{figure}

\begin{theorem}[Temperature Transition \cite{Shen2024CVPR}]
Assume that for some \(\beta_1,\beta_2>0\), the solution
\(S^{\beta_1}_{\mathbb{U}^\mathbf{p}_\mathbf{q}} = \Gamma^{{KL}}_{\mathbb{U}^\mathbf{p}_\mathbf{q}}(\expo(\beta_1 S))\) has been computed.
Then the solution at inverse temperature \(\beta_2\) is obtained by
\[
S^{\beta_2}_{\mathbb{U}^\mathbf{p}_\mathbf{q}}
= \Gamma^{{KL}}_{\mathbb{U}^\mathbf{p}_\mathbf{q}}\Bigl(
\bigl(S^{\beta_1}_{\mathbb{U}^\mathbf{p}_\mathbf{q}}\bigr)^{\odot (\beta_2/\beta_1)}
\Bigr),
\]
where \(S^{\odot \beta}\) denotes the element-wise power, i.e. \((S^{\odot \beta})_{ij} = (S_{ij})^\beta\).
\end{theorem}

The transition theorem converts independent Sinkhorn solves into warm-started updates. Without transition, each new $\beta$ is solved from a cold start, with cost
\(
O(T_{\mathrm{Sinkhorn}}\,mn).
\). With transition, moving from $\beta_1$ to $\beta_2$ requires one Hadamard power, costing $O(mn)$, followed by Sinkhorn scaling with cost
\(
O(T'_{\mathrm{Sinkhorn}}\,mn).
\)
The transition does not improve the worst-case contraction factor, but reuses the previous scaling information and can reduce the number of Sinkhorn iterations in practice. Experiments in \cite{Shen2024CVPR} report approximately $85\%$ computational savings for random matrices with $n=2000$.

This theorem is supported by two important lemmas.
\begin{lemma}[KL-projection invariance under diagonal scaling \cite{Shen2024CVPR}]
For any marginals \(\mathbf{p}\in\mathbb{R}^m_+\) and \(\mathbf{q}\in\mathbb{R}^n_+\), any scaling vectors \(\mathbf{r}\in\mathbb{R}^m_{++}\) and \(\mathbf{c}\in\mathbb{R}^n_{++}\), and any kernel \(K\in\mathbb{R}^{m\times n}_{++}\), the KL projection onto \(\mathbb{U}^{\mathbf{p}}_{\mathbf{q}}\) satisfies
\[
\Gamma^{{KL}}_{\mathbb{U}^{\mathbf{p}}_{\mathbf{q}}}\bigl(D(\mathbf{r}) K D(\mathbf{c})\bigr)
=
\Gamma^{{KL}}_{\mathbb{U}^{\mathbf{p}}_{\mathbf{q}}}(K).
\]
\end{lemma}

\noindent This lemma was proved in \cite{Shen2024CVPR} using the uniqueness theory of Sinkhorn scaling. Here we provide an alternative proof from the perspective of KL projection. For any \(P\in\mathbb{U}^{\mathbf{p}}_{\mathbf{q}}\), by the definition of the generalized KL divergence, we have
\begin{proof}
\[
\begin{aligned}
D_{{KL}}\bigl(P \,\|\, D(\mathbf{r})KD(\mathbf{c})\bigr)
&= \sum_{i,j} \left( P_{ij}\log\frac{P_{ij}}{r_i K_{ij} c_j} - P_{ij} + r_i K_{ij} c_j \right) \\
&= \sum_{i,j} \left( P_{ij}\log\frac{P_{ij}}{K_{ij}} - P_{ij} + K_{ij} \right) \\
&\quad - \sum_{i,j} P_{ij}(\log r_i + \log c_j)
  + \sum_{i,j} (r_i c_j - 1) K_{ij}.
\end{aligned}
\]
Since \(P\in\mathbb{U}^{\mathbf{p}}_{\mathbf{q}}\), we have \(\sum_j P_{ij}=p_i\) and \(\sum_i P_{ij}=q_j\), so
\[
\sum_{i,j} P_{ij}(\log r_i + \log c_j)
= \sum_i p_i \log r_i + \sum_j q_j \log c_j,
\]
which is independent of \(P\). The remaining term \(\sum_{i,j}(r_i c_j-1)K_{ij}\) is also constant with respect to \(P\). Hence
\[
D_{{KL}}\bigl(P \,\|\, D(\mathbf{r})KD(\mathbf{c})\bigr)
= D_{{KL}}(P \| K) + \text{constant}.
\]
Therefore, the minimizer over \(P\in\mathbb{U}^{\mathbf{p}}_{\mathbf{q}}\) is identical for both divergences, which proves the claimed invariance.
\end{proof}
The Lemma mathematically characterizes the invariance of the KL-divergence projection operator \(\Gamma^{{KL}}_{\mathbb{U}^{\mathbf{p}}_{\mathbf{q}}}\) under row‑column scaling transformations. Algebraically, the lemma establishes an equivalence relation: matrices connected via independent row‑and‑column scalings form a equivalent class, and every member of that class projects to the identical point under the KL divergence with fixed marginals. This provides a parametric redundancy for numerical algorithms such as Sinkhorn iterations, which can be exploited to improve numerical stability and accelerate convergence. 

\begin{lemma}[KL-projection Composition \cite{Shen2024CVPR}]
For any matrices \(K_1, K_2 \in \mathbb{R}_{++}^{m\times n}\), the KL projection satisfies the following composition property:
\[
\Gamma^{{KL}}_{{\mathbb{U}^\mathbf{p}_\mathbf{q}}}\Bigl(
   \Gamma^{{KL}}_{{\mathbb{U}^\mathbf{p}_\mathbf{q}}}(K_1) \odot K_2
\Bigr)=\Gamma^{{KL}}_{{\mathbb{U}^\mathbf{p}_\mathbf{q}}}(K_1 \odot K_2)
,
\]
where \(\odot\) denotes the Hadamard (element-wise) product.
\end{lemma}
\noindent The composition law states that, when projecting the Hadamard product $K_1\odot K_2$, one may replace $K_1$ by its KL-projection without changing the final result. This is because the two resulting products differ only by positive row and column scalings, which do not affect the projection. Figure \ref{fig:kl-properties} visualizes the theoretical result established above.

This composition law underpins the temperature transition theorem. Indeed, the Gibbs kernel at \(\beta_2\) factorizes as
\[
\expo(\beta_2 S) = \expo(\beta_1 S) \odot \expo((\beta_2 - \beta_1) S).
\]
Applying the composition lemma with \(K_1=\expo(\beta_1 S)\) and \(K_2=\expo((\beta_2-\beta_1)S)\) yields
\[
S^{\beta_2}_{\mathbb{U}^{\mathbf{p}}_{\mathbf{q}}}
= \Gamma^{{KL}}_{\mathbb{U}^{\mathbf{p}}_{\mathbf{q}}}\Bigl(
   S^{\beta_1}_{\mathbb{U}^{\mathbf{p}}_{\mathbf{q}}} \odot \expo((\beta_2 - \beta_1) S)
\Bigr).
\]
Since \(S^{\beta_1}_{\mathbb{U}^{\mathbf{p}}_{\mathbf{q}}} = \Gamma^{{KL}}_{\mathbb{U}^{\mathbf{p}}_{\mathbf{q}}}(\expo(\beta_1 S)) = D(\mathbf{u})\expo(\beta_1 S)D(\mathbf{v})\), the invariance lemma implies that the diagonal scalings do not affect the projection. Consequently,
\[
S^{\beta_2}_{\mathbb{U}^{\mathbf{p}}_{\mathbf{q}}}
= \Gamma^{{KL}}_{\mathbb{U}^{\mathbf{p}}_{\mathbf{q}}}\Bigl(
   \bigl(S^{\beta_1}_{\mathbb{U}^{\mathbf{p}}_{\mathbf{q}}}\bigr)^{\odot (\beta_2/\beta_1)}
\Bigr),
\]
which recovers the temperature transition formula. Thus, the composition law provides the theoretical basis for efficient warm-start continuation. 

\begin{figure}
    \centering
    \includegraphics[width=0.9\linewidth]{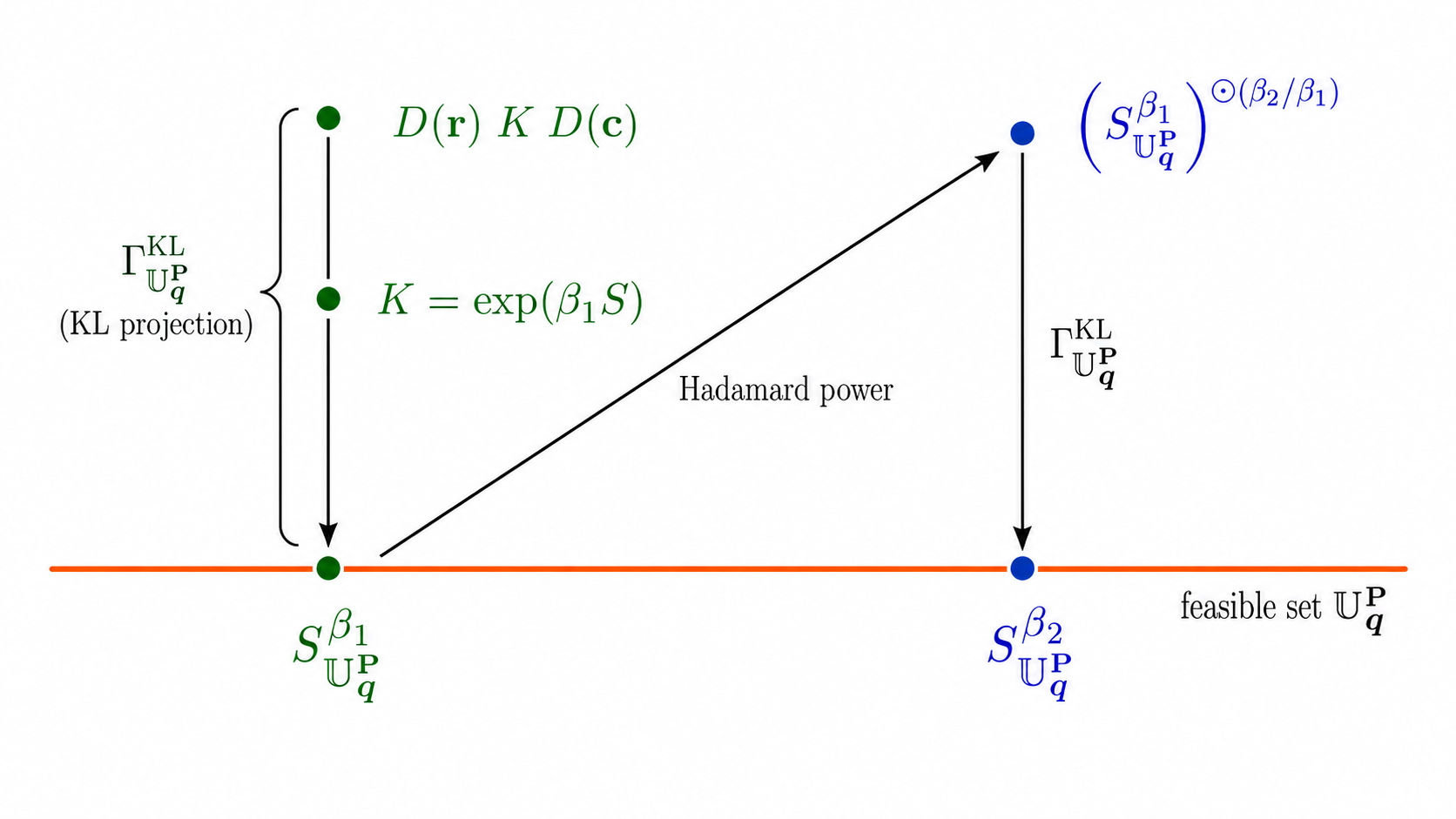}
    \caption{Properties of the KL projection in entropic optimal transport.}
    \label{fig:kl-properties}
\end{figure}

\subsection{Assignment Error and Scale-Invariant Parameterization}

A fundamental question is: how close is the entropy-regularized solution to the true linear assignment optimum? Define the {normalized assignment error}:
\begin{equation}
\Delta_\beta^S = \frac{1}{n}\bigl(\langle S^\infty_{\mathbb{U}}, S \rangle - \langle S^\beta_{\mathbb{U}}, S \rangle\bigr),
\end{equation}
where $S^\infty_{\mathbb{U}}=\lim\limits_{\beta \to \infty} S^\beta_{\mathbb{U}}$ is the limiting linear assignment solution. Since $-n\log n \leq \mathcal{H}(X) \leq 0$, we immediately obtain:
\begin{equation}
\Delta_\beta^S \leq \frac{\log n}{\beta}.
\end{equation}
This bound reveals a critical scaling issue: fixing $\beta$ yields a worst-case error bound that grows logarithmically with $n$. The remedy is the scale-invariant parameterization $\beta = \gamma \log n$, which gives:
\begin{equation}
\Delta_\beta^S \leq \frac{1}{\gamma}.
\end{equation}

Now the error bound depends only on the dimensionless parameter $\gamma$, decoupling precision from problem size. This insight underlies the {CSGO} (Constrained-Softassign Gradient Optimization) \cite{Shen2026CSGO}, which uses $\gamma \log n$ as the default regularization schedule for large-scale graph matching.

\section{The Frobenius/Euclidean Path}

\subsection{Frobenius-Regularized Assignment}

The Frobenius-regularized assignment problem (FRA) is defined as the unique maximizer:
\begin{equation}
S_{\mathbb{U}}^\theta = \arg \max_{X \in \mathbb{U}} \  \langle X, S \rangle - \frac{1}{\theta}\|X\|_F^2 .
\end{equation}
\noindent By completing the square, this is equivalent to a Euclidean projection:
\begin{equation}
S_{\mathbb{U}}^\theta = \argmin_{X \in \mathbb{U}} \left\|X - \frac{\theta}{2}S\right\|_F^2 = \Gamma^{F}_{\mathbb{U}}\!\left(\frac{\theta}{2}S\right).
\end{equation}
Here we use the notation $\Gamma^{F}_{\mathbb{U}}$ for the Euclidean projection operator, which is well-defined since the projection onto a closed convex set is unique. The parameter $\theta$ has an intuitive interpretation: small $\theta$ gives diffuse, conservative directions near the center of $\mathbb{U}$; large $\theta$ drives the projection toward the optimal face of permutation matrices. Figure~\ref{fig:fra-theta-path} illustrates this parameter path.

\begin{figure}[t]
\centering
\begin{minipage}[t]{0.32\textwidth}
\centering
\includegraphics[width=\textwidth]{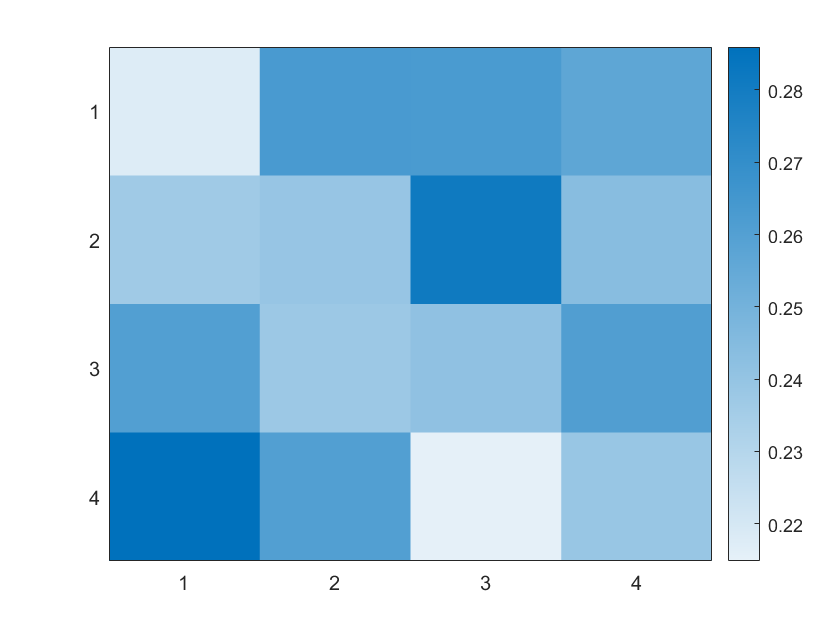}\\
\small $\theta=0.1$
\end{minipage}\hfill
\begin{minipage}[t]{0.32\textwidth}
\centering
\includegraphics[width=\textwidth]{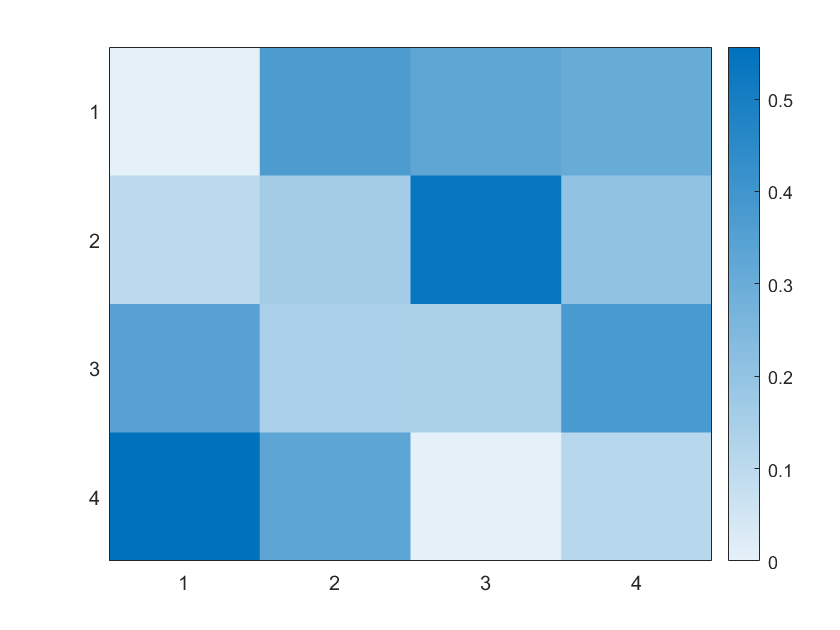}\\
\small $\theta=1$
\end{minipage}\hfill
\begin{minipage}[t]{0.32\textwidth}
\centering
\includegraphics[width=\textwidth]{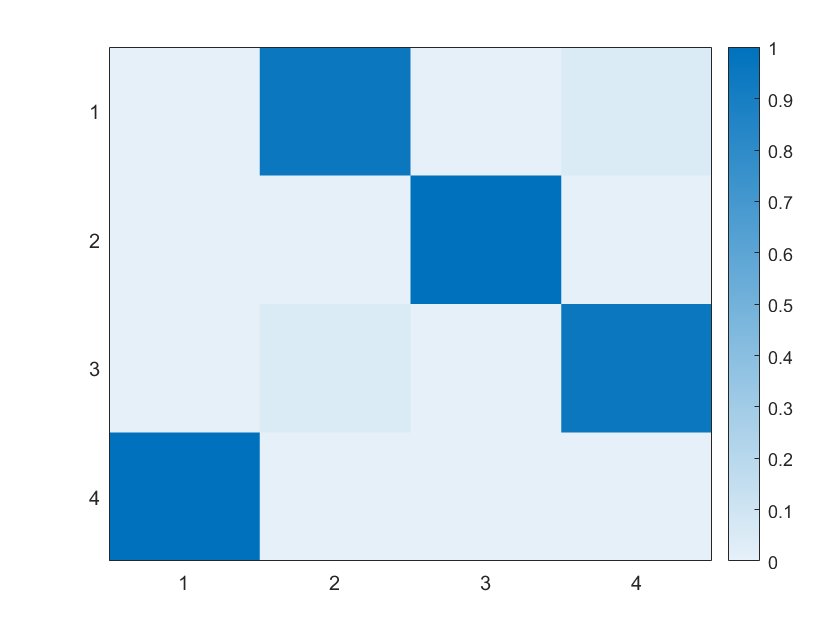}\\
\small $\theta=10$
\end{minipage}
\caption{Visualization of $S_{\mathbb{U}}^\theta$ under different values of $\theta$. The color of each cell represents the matrix entry, with darker shades indicating larger values \cite{Shen2025NeurIPS}.}
\label{fig:fra-theta-path}
\end{figure}
The limiting behavior is precisely characterized:

\begin{theorem}[Limit of FRA \cite{Shen2025NeurIPS}]
\label{thm:fra_limit}
Let $S \in \R_+^{n \times n}$ and let $\mathcal{F}_S$ be the convex hull of optimal permutation matrices for the linear assignment problem with score matrix $S$. Then:
\begin{itemize}
\item As $\theta \to 0$: $S_{\mathbb{U}}^\theta \to \frac{1}{n}\mathbf{1}_n\mathbf{1}_n^\top$ (the center of $\mathbb{U}$).
\item As $\theta \to \infty$: $S_{\mathbb{U}}^\theta$ converges to the unique point in $\mathcal{F}_S$ with minimal Frobenius norm; if the optimal permutation is unique, this limit is exactly that permutation matrix.
\end{itemize}
\end{theorem}
\noindent This theorem reveals a subtle but important distinction from the entropy path: when multiple optimal permutations exist, the Frobenius limit selects the unique minimum-Frobenius-norm point of the optimal face. The two paths may converge to different points on the same optimal face.

The normalized assignment error satisfies an analogous bound:
\begin{equation}
 \frac{1}{n}\bigl(\langle S_{\mathbb{U}}^\infty, S \rangle - \langle S_{\mathbb{U}}^\theta, S \rangle\bigr) \leq \frac{1}{\theta},
\end{equation}
showing that $\theta$ simultaneously controls directional sharpness and regularization bias.

\subsection{Scaling Doubly Stochastic Normalization (SDSN)}

The Euclidean projection onto $\mathbb{U}$ has no simple closed-form expression, but it can be approximated efficiently by {alternating projections} between two convex sets:
\begin{enumerate}[leftmargin=2em]
\item $\mathbb{A} = \{Y \in \R^{n \times n} : Y\mathbf{1}_n = \mathbf{1}_n,\, Y^\top\mathbf{1}_n = \mathbf{1}_n\}$ (affine subspace of row/column sum constraints);
\item $\R_+^{n \times n}$ (nonnegative orthant).
\end{enumerate}

The SDSN iterates:
\begin{equation}
Y^{(k+1)} = \mathcal{P}_2 (\mathcal{P}_1(Y^{(k)})),
\end{equation}
where:
\begin{gather}
\mathcal{P}_1(Y)
=
Y+\left(
\frac{n+\mathbf{1}_n^\top Y\mathbf{1}_n}{n^2}I_n
-\frac{1}{n}Y
\right)
\mathbf{1}_n\mathbf{1}_n^\top
-\frac{1}{n}\mathbf{1}_n\mathbf{1}_n^\top Y,
\\
\mathcal{P}_2(Y)
=
[Y]_+
=
\frac{Y+|Y|}{2}.
\end{gather}

\noindent Here $\mathcal P_1$ is the Frobenius projection onto the affine subspace $\mathbb A$, and $\mathcal P_2$ is the Frobenius projection onto the nonnegative orthant $\R_+^{n\times n}$. The convergence of this alternating projection scheme follows from the classical theory of convex feasibility \cite{Bauschke1996}.

Three theoretical properties make SDSN particularly attractive for large-scale computation:

\begin{enumerate}[leftmargin=2em]
\item \textbf{Linear convergence}: For input $S$ and error tolerance $\eta$, the number of iterations needed is
\begin{equation}
k = \left\lceil \frac{\ln\bigl(\eta / [\theta(\|S\|_F + n)]\bigr)}{\ln c} \right\rceil,
\end{equation}
where $c \in (0,1)$ is the linear convergence factor. Larger $\theta$ (sharper projections) require more iterations.

\item \textbf{Dimension-normalized stopping criterion}: The quantity $\zeta(Y^{(k)}) = \frac{1}{n}\mathbf{1}_n^\top Y^{(k)}\mathbf{1}_n - 1$ measures total mass deviation after the nonnegativity projection. It provides a cheap, scale-invariant proxy for feasibility error.

\item \textbf{Error damping}: Low-precision arithmetic introduces rounding or truncation errors at each step. Although $\mathcal{P}_1$ is an affine projection, it acts linearly on these perturbations, mapping them into the subspace of matrices with zero row and column sums. Under the contraction properties of the subsequent SDSN iterations, previously introduced errors are progressively attenuated, enabling mixed-precision acceleration with a controlled loss of accuracy \cite{Shen2025NeurIPS}.
\end{enumerate}

\section{Graph Matching: Two Geometries, One Nonconvex Outer Problem}

\subsection{The Koopmans--Beckmann Formulation}

Graph matching seeks to align two undirected graphs $G = (V, A)$ and $\tilde{G} = (\tilde{V}, \tilde{A})$ with adjacency matrices $A, \tilde{A} \in \R^{n \times n}$ and vertex similarity matrix $Q \in \R^{n \times n}$. The classical Koopmans--Beckmann formulation \cite{Beckmann1957, Lawler1963} is:
\begin{equation}
\max_{M \in \Pi} \Phi(M) = \frac{1}{2}\tr(M^\top A M \tilde{A}) + \lambda \tr(M^\top Q),
\label{eq.Object}
\end{equation}
where $\Pi$ is the set of permutation matrices. This is the quadratic assignment problem (QAP), which is NP-hard \cite{garey1979computers}.

The standard relaxation replaces the set of permutation matrices with the Birkhoff polytope \(\mathbb{U}\) of doubly stochastic matrices. Although this relaxation makes the feasible set convex, the resulting maximization problem remains nonconvex because the objective is generally nonconcave. A natural approach is the relaxed gradient-projection method \cite{Lu2016}:

\begin{equation}
M^{(t+1)} = (1-\alpha_t)M^{(t)} + \alpha_t D^{(t)},
\end{equation}
\begin{equation}
 D^{(t)} = \Gamma^F_{\mathbb{U}}(\nabla \Phi(M^{(t)})) = \Gamma^F_{\mathbb{U}}(A M^{(t)} \tilde{A} + \lambda Q).
\end{equation}
This alternates between taking a step along the gradient of the quadratic objective and projecting the resulting point onto \(\mathbb{U}\).

\subsection{The Hidden Regularization in Projection Gradient Steps}

The solution to the quadratic assignment problem~\eqref{eq.Object} also maximizes \( w \Phi(M) \), where \( w \) is a positive scaling constant:
\begin{equation}
    \arg \max_{M \in \Pi} w\Phi(M) = \arg \max_{M \in \Pi} \Phi(M)
\end{equation}
However, the doubly stochastic projection \(\Gamma^F_{\mathbb{U}}(\cdot)\) fails to preserve this scale-invariant property:
\begin{equation}
\Gamma^F_{\mathbb{U}}(S) \neq \Gamma^F_{\mathbb{U}}(wS), \quad \text{for } S \in \mathbb{R}^{n \times n}_+.
\end{equation}
Since the objective function is non-convex, this sensitivity can cause the projection-based algorithm to converge to different points under different scalings.

To understand the precise mechanism behind this scale dependence, we examine how the projection operator behaves under a scaled input. The following lemma provides the key characterization.
\begin{lemma}[Projection as Implicit Regularized Assignment \cite{Shen2025NeurIPS}]
\label{lem:projection}
For any scale factor $w > 0$ and score matrix $S$:
\begin{equation}
\Gamma^{F}_{\mathbb{U}}(wS) = \argmin_{X \in \mathbb{U}} \|X - wS\|_F^2 = \arg \max_{X \in \mathbb{U}}\   \langle X, S \rangle - \frac{1}{2w}\|X\|_F^2 .
\end{equation}
\end{lemma}
This reveals that the projection step is exactly a Frobenius-regularized assignment problem with parameter $\theta = 2w$. The scale factor $w$ controls the tradeoff between aligning with the gradient (high assignment score) and maintaining a diffuse solution close to the barycenter (low Frobenius norm). 

If the gradient magnitude changes during optimization, the relative balance between the assignment score and the Frobenius regularization also changes, causing the effective regularization parameter to vary implicitly across iterations. This scale dependence may alter the sharpness of the projected solution and distort the resulting optimization trajectory. To decouple the regularization strength from the gradient magnitude, we normalize the gradient before projection and treat \(\theta\) as an explicit, controllable algorithmic parameter.

\subsection{FRAM: The Frobenius Route}
Frobenius-Regularized Assignment Matching (FRAM) \cite{Shen2025NeurIPS} implements this insight explicitly:
\begin{enumerate}[leftmargin=2em]
\item Normalize the gradient: $\hat{S}^{(t)} = \nabla \Phi(M^{(t)})/{\max\limits_{i,j}  \left[\nabla \Phi(M^{(t)})\right]_{ij}}$;
\item Approximately compute the FRA feasible target: $D^{(t)} = \Gamma^{F}_{\mathbb{U}}(\frac{\theta}{2}\hat{S}^{(t)})$ via SDSN;
\item Projected-gradient update: $M^{(t+1)} = (1-\alpha_t)M^{(t)} + \alpha_t D^{(t)}$.
\end{enumerate}

The normalization ensures that the sharpness of the FRA solution is controlled by $\theta$, not by gradient magnitude fluctuations. The per-iteration cost is dominated by the matrix multiplication $A M^{(t)} \tilde{A}$ (for gradient computation) and the SDSN iterations (for FRA projection). The algorithm incurs a time complexity of $O(n^3 + \ell n^2)$ and a space complexity of $O(n^2)$, where $\ell$ is the number of SDSN inner iterations.

A key practical advantage is {mixed-precision acceleration}: exploiting the error-damping property of SDSN, FRAM can compute the gradient and the inner projection in low-precision (e.g., TF32 and FP32) while maintaining FP64 accuracy in the outer iterations. On benchmark problems, this yields speedups of up to $370\times$ relative to CPU FP64 implementations (shown in Figure \ref{fig:mix}), with accuracy degradation below 0.2\% \cite{Shen2025NeurIPS}.

\begin{figure}[htbp]
\centering
\includegraphics[width=0.9\linewidth]{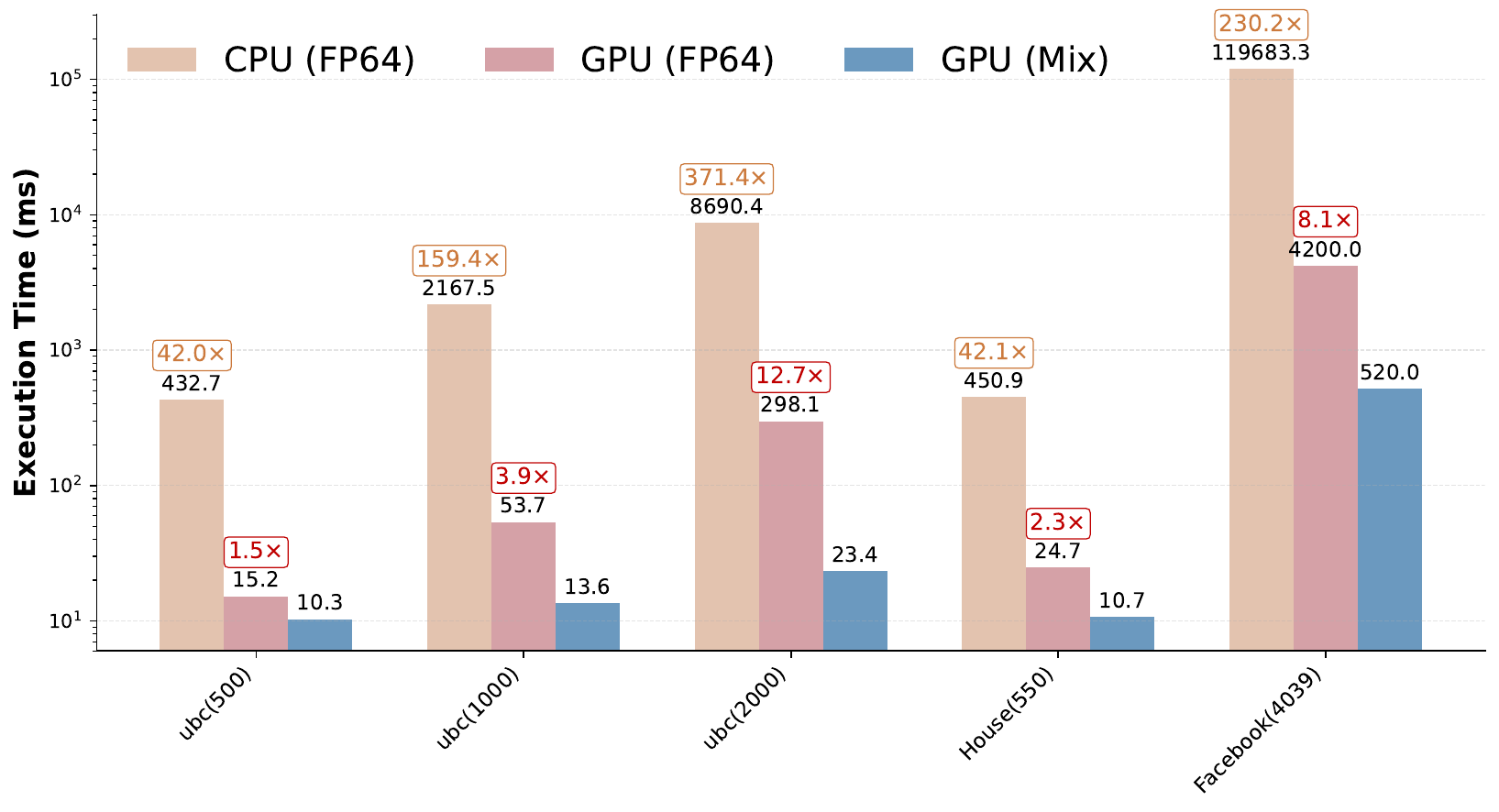}
\caption{Runtime comparison of FRAM under different precisions and processors. The vertical axis shows the runtime (log scale), and the horizontal axis lists the datasets and their sizes. The numbers in red boxes indicate the speedup of GPU mixed precision over GPU double precision, while those in yellow boxes indicate the speedup over CPU double precision.}
\label{fig:mix}
\end{figure}

\subsection{CSGO and ASM: The Entropy Route}
\label{sec:asm}

The entropy route follows the same projected-gradient update but with softassign directions:
\begin{equation}
D^{(t)} = S^{\beta_t}_{\mathbb{U}} = \Gamma^{\KL}_{\mathbb{U}}\bigl(\expo(\beta_t \hat{S}^{(t)})\bigr).
\end{equation}

\textbf{CSGO} (Constrained-Softassign Gradient Optimization) \cite{Shen2026CSGO} establishes the $\beta_t = \gamma_t \log n$ parameterization, ensuring that the normalized assignment error bound $1/\gamma_t$ is independent of problem size.

\textbf{ASM} (Adaptive Softassign Matching) \cite{Shen2024CVPR} goes further by introducing an adaptive annealing criterion. Rather than fixing a temperature schedule, ASM monitors the marginal benefit of increasing $\beta$:
\begin{equation}
\beta_\eta(S) = \min\left\{ \beta \geq \beta_0 : \|S^{\beta+\Delta\beta}_{\mathbb{U}} - S^\beta_{\mathbb{U}}\|_1 \leq \eta \right\}, \quad \Delta\beta = \log n.
\end{equation}

The stopping condition checks whether the directional change justifies the computational cost of a sharper coupling. Crucially, ASM leverages the transition theorem to evaluate candidate temperatures without restarting Sinkhorn iterations:
\begin{equation}
S^{\beta_2}_{\mathbb{U}} = \Gamma^{\KL}_{\mathbb{U}}\Bigl((S^{\beta_1}_{\mathbb{U}})^{\odot(\beta_2/\beta_1)}\Bigr).
\end{equation}

This transition recursively initializes each candidate coupling from the previous one. Although each transition still requires a Hadamard power followed by Sinkhorn matrix scaling, it can substantially reduce the number and cost of Sinkhorn iterations required during adaptive exploration. The paper \cite{Shen2024CVPR} reports that, in experiments on random matrices with \(n=2000\), this recursive computation reduced the computational cost by approximately \(85\%\).

Figure~\ref{fig:performance} demonstrates the representative performance of ASM in a large-scale graph matching task, yeast Protein-Protein Interaction (PPI) network alignment.
\begin{figure}[h]
\centering
\includegraphics[width=0.75\linewidth]{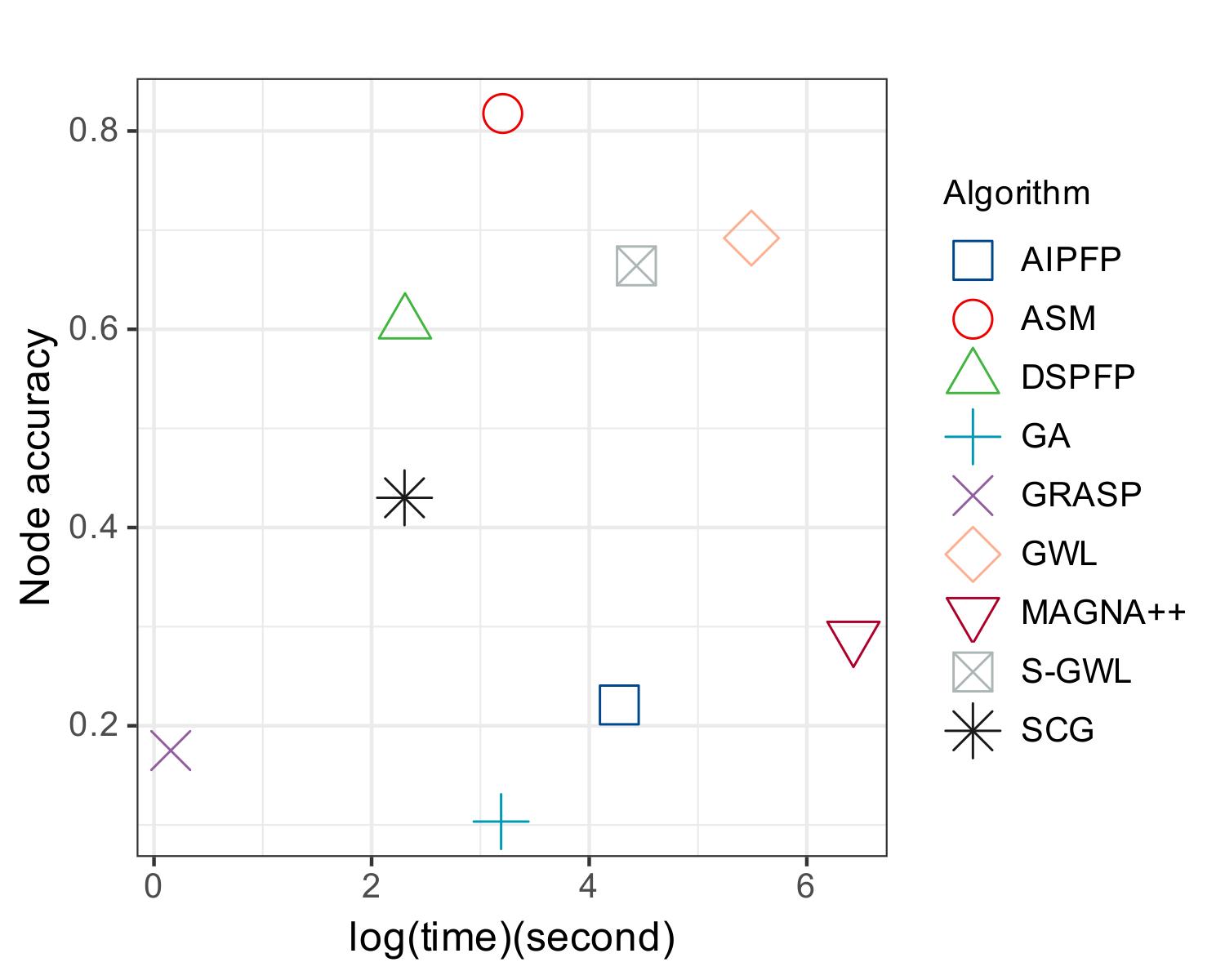}
\caption{Matching accuracy versus runtime on the yeast PPI network (1,004 proteins, 4,920 high-confidence interactions) \cite{Shen2024CVPR}.}
\label{fig:performance}
\end{figure}
\section{Optimal Transport: From Direction to Destination}

In graph matching, the regularized coupling map generates a {search direction} within a projected-gradient iteration; the final output is a discrete matching, and the coupling is merely an intermediate computational object. In optimal transport, the coupling is the final answer---it must satisfy the marginal constraints exactly, and its quality is measured by how well it approximates the true transport plan.

This difference in computational roles leads to different numerical requirements. Graph-matching algorithms may tolerate moderately approximate intermediate couplings or search directions, provided that they still produce effective progress toward a high-quality discrete matching. Optimal-transport solvers, however, generally require substantially tighter marginal feasibility and numerical accuracy in the final coupling.

\subsection{IP-EOT: Inexact Proximal Entropic Optimal Transport}
\label{sec:ipeot}

For large regularization parameter $\beta$, directly computing $\expo(\beta S)$ may cause overflow or numerical instability. The {inexact proximal entropic optimal transport} (IP-EOT) method \cite{Shen2026OpenReview} addresses this by {proximal continuation}: decomposing a large $\beta$ into smaller increments $\Delta\beta_t$ and approximately solving a sequence of proximal subproblems.

The proximal point method \cite{xie2020fast} approximates the solution of an optimal transport problem by solving a sequence of simpler entropic optimal transport problems:
\begin{equation}
P^{(t+1)} = \arg\max_{P \in \mathbb{U}^\mathbf{p}_\mathbf{q}} \langle S, P \rangle - \frac{1}{\Delta\beta_t} D_{{KL}}(P \| P^{(t)}).
\end{equation}
Here, the KL term serves as a proximal regularizer that keeps the new coupling $P^{(t+1)}$ close to the previous iterate $P^{(t)}$ in an
information-geometric sense, which prevents abrupt changes in the transport plan. This is equivalent to a Hadamard update followed by a KL projection:
\begin{equation}
P^{(t+1)} = \Gamma^{\KL}_{\mathbb{U}^\mathbf{p}_\mathbf{q}}\bigl(P^{(t)} \odot \expo(\Delta\beta_t S)\bigr).
\end{equation}
To reduce the computational cost, the authors \cite{xie2020fast} proposed an inexact version in which each projection $\Gamma^{\KL}_{\mathbb{U}^\mathbf{p}_\mathbf{q}}$ is performed using only a single Sinkhorn iteration, and proved that it still converges to the optimal solution. However, the proximal point method has not been applicable to entropic optimal transport because the trajectory of $\beta$ during the iterations remains unclear. Our work seeks to address this issue in the following proposition.

\begin{proposition}[Trajectory of the regularization parameter \cite{Shen2026OpenReview}]
\label{prop:same-kernel}
Let $\beta_t = \sum_{s=0}^{t-1} \Delta\beta_s$ and initialize $P^{(0)}$ as $\mathbf{1}_m \mathbf{1}_n^\top$. If each proximal step is solved exactly, then:
\begin{equation}
P^{(t)} = \Gamma^{\KL}_{\mathbb{U}^\mathbf{p}_\mathbf{q}}\bigl(P^{(0)} \odot \expo(\beta_t S)\bigr) = \Gamma^{\KL}_{\mathbb{U}^{\mathbf{p}}_{\mathbf{q}}}\bigl(\expo(\beta_t S)\bigr).
\end{equation}
\end{proposition}
\noindent The key insight is that proximal decomposition does {not} alter the score matrix $S$; it merely decomposes the Gibbs kernel into cumulative Hadamard factors. Each KL projection introduces diagonal scalings that are absorbed by the next projection, preserving the same-kernel condition. This proposition enables the proximal point method to be applied to entropic optimal transport.

In practice, each stage is solved inexactly with limited Sinkhorn iterations, which yields the Inexact Proximal point method for Entropic Optimal Transport (IP-EOT) \cite{Shen2026OpenReview}. The following theorem guarantees that exactness can be restored:

\begin{theorem}[Inexact-to-exact transition \cite{Shen2026OpenReview}]
\label{thm:inexact-exact}
Let 
$$\hat{S}^\beta = \diag(\hat{\mathbf{u}}) \expo(\beta S) \diag(\hat{\mathbf{v}})$$ 
be any positive diagonal scaling of the Gibbs kernel. Then a single exact KL projection recovers the target solution:
\begin{equation}
\Gamma^{\KL}_{\mathbb{U}^\mathbf{p}_\mathbf{q}}(\hat{S}^\beta) = \Gamma^{\KL}_{\mathbb{U}^\mathbf{p}_\mathbf{q}}\bigl(\expo(\beta S)\bigr) = S^\beta_{\mathbb{U}^\mathbf{p}_\mathbf{q}}.
\end{equation}
\end{theorem}
\noindent This theorem is the mathematical foundation for IP-EOT as a {warm-start strategy}: the proximal stages need not be solved to high precision; they need only preserve the same-kernel structure, after which a single exact projection yields the final answer.     
\begin{figure}[h]
\centering
\includegraphics[width=0.8\textwidth]{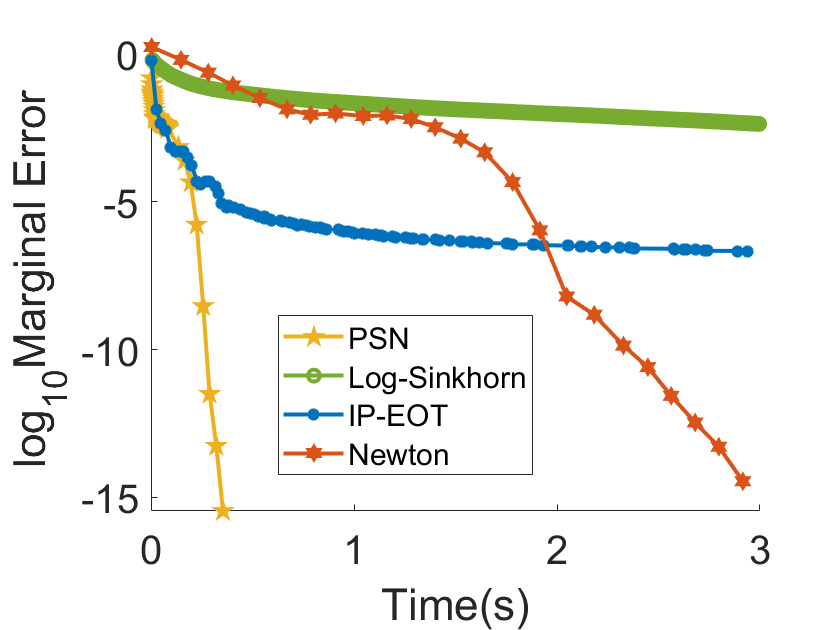}
\caption{Performance of IP-EOT and PSN.}
\label{fig.psn1}
\end{figure}
Figure~\ref{fig.psn1} illustrates the role of this proximal warm-start via marginal residual curves under the given experimental setup. The proximal path first provides an approximation that satisfies the same-kernel condition, causing the marginal residuals to decrease in the early stage; if a tighter tolerance is required, an exact KL projection such as Newton's method or Sinkhorn scaling is still needed afterward.

\subsection{PSN: Proximal--Sinkhorn--Newton Solver}

The {Proximal--Sinkhorn--Newton} (PSN) solver \cite{Shen2026OpenReview} orchestrates three algorithmic components:
\begin{enumerate}[leftmargin=2em]
\item \textbf{IP-EOT warm-start}: Construct a same-kernel approximation via proximal continuation.
\item \textbf{Sinkhorn iteration}: For moderate $\beta$, continue Sinkhorn scaling to drive marginal residuals below tolerance.
\item \textbf{Newton refinement}: For large $\beta$ where the coupling is nearly sparse, switch to Newton's method on the marginal equations for local quadratic convergence.
\end{enumerate}
Figure~\ref{fig:psn-architecture} illustrates the regularization strength scheduling logic in this workflow.
\begin{figure}[h]
\tikzset{every picture/.style={line width=0.4pt}} 

\begin{tikzpicture}[x=0.75pt,y=0.75pt,yscale=-1,scale=0.8]

\draw [color={rgb, 255:red, 208; green, 2; blue, 27 }  ,draw opacity=1 ][line width=3]    (136.1,88.14) -- (177.97,36.04) ;
\draw [shift={(181.1,32.14)}, rotate = 128.78] [color={rgb, 255:red, 208; green, 2; blue, 27 }  ,draw opacity=1 ][line width=3]    (14.54,-4.38) .. controls (9.25,-1.86) and (4.4,-0.4) .. (0,0) .. controls (4.4,0.4) and (9.25,1.86) .. (14.54,4.38)   ;
\draw [color={rgb, 255:red, 74; green, 144; blue, 226 }  ,draw opacity=1 ][line width=3]    (185.5,37.5) -- (185.85,97.93) ;
\draw [shift={(185.88,102.93)}, rotate = 269.67] [color={rgb, 255:red, 74; green, 144; blue, 226 }  ,draw opacity=1 ][line width=3]    (14.54,-4.38) .. controls (9.25,-1.86) and (4.4,-0.4) .. (0,0) .. controls (4.4,0.4) and (9.25,1.86) .. (14.54,4.38)   ;
\draw  [color={rgb, 255:red, 245; green, 166; blue, 35 }  ,draw opacity=1 ][line width=6] [line join = round][line cap = round] (103.91,109.02) .. controls (108.54,108.53) and (115,121.91) .. (118.69,125.5) .. controls (129.74,136.22) and (139.8,146.5) .. (151.99,155.67) .. controls (187.96,182.75) and (239.57,198.14) .. (281.37,201.09) .. controls (325.69,204.21) and (417.59,210.01) .. (457.67,194.79) .. controls (476.91,187.48) and (497.82,177.51) .. (516.24,168.28) .. controls (521.55,165.63) and (526.37,160.06) .. (531.17,157.12) .. controls (533.58,155.65) and (536.85,156.05) .. (538.88,153.92) .. controls (541.54,151.13) and (543.51,146.23) .. (547.09,145.86) ;
\draw  [color={rgb, 255:red, 65; green, 117; blue, 5 }  ,draw opacity=1 ][fill={rgb, 255:red, 65; green, 117; blue, 5 }  ,fill opacity=1 ] (515.33,165.65) .. controls (514.08,162.4) and (519.68,157.23) .. (527.82,154.09) .. controls (535.97,150.95) and (543.59,151.03) .. (544.84,154.27) .. controls (546.09,157.51) and (540.5,162.69) .. (532.35,165.83) .. controls (524.2,168.97) and (516.59,168.89) .. (515.33,165.65) -- cycle ;
\draw  [color={rgb, 255:red, 0; green, 0; blue, 0 }  ,draw opacity=1 ][fill={rgb, 255:red, 0; green, 0; blue, 0 }  ,fill opacity=1 ] (126.6,93.64) .. controls (126.6,90.6) and (129.06,88.14) .. (132.1,88.14) .. controls (135.14,88.14) and (137.6,90.6) .. (137.6,93.64) .. controls (137.6,96.68) and (135.14,99.14) .. (132.1,99.14) .. controls (129.06,99.14) and (126.6,96.68) .. (126.6,93.64) -- cycle ;
\draw [color={rgb, 255:red, 155; green, 155; blue, 155 }  ,draw opacity=1 ]   (138.77,208.15) -- (151.07,175.02) ;
\draw [shift={(151.77,173.15)}, rotate = 110.38] [color={rgb, 255:red, 155; green, 155; blue, 155 }  ,draw opacity=1 ][line width=0.75]    (19.67,-5.92) .. controls (12.51,-2.51) and (5.95,-0.54) .. (0,0) .. controls (5.95,0.54) and (12.51,2.51) .. (19.67,5.92)   ;
\draw [color={rgb, 255:red, 155; green, 155; blue, 155 }  ,draw opacity=1 ]   (562,190) -- (547,170.72) ;
\draw [shift={(545.77,169.15)}, rotate = 52.1] [color={rgb, 255:red, 155; green, 155; blue, 155 }  ,draw opacity=1 ][line width=0.75]    (19.67,-5.92) .. controls (12.51,-2.51) and (5.95,-0.54) .. (0,0) .. controls (5.95,0.54) and (12.51,2.51) .. (19.67,5.92)   ;
\draw [color={rgb, 255:red, 208; green, 2; blue, 27 }  ,draw opacity=1 ][line width=3]    (191.1,102.14) -- (232.97,50.04) ;
\draw [shift={(236.1,46.14)}, rotate = 128.78] [color={rgb, 255:red, 208; green, 2; blue, 27 }  ,draw opacity=1 ][line width=3]    (14.54,-4.38) .. controls (9.25,-1.86) and (4.4,-0.4) .. (0,0) .. controls (4.4,0.4) and (9.25,1.86) .. (14.54,4.38)   ;
\draw [color={rgb, 255:red, 74; green, 144; blue, 226 }  ,draw opacity=1 ][line width=3]    (240.5,51.5) -- (240.85,111.93) ;
\draw [shift={(240.88,116.93)}, rotate = 269.67] [color={rgb, 255:red, 74; green, 144; blue, 226 }  ,draw opacity=1 ][line width=3]    (14.54,-4.38) .. controls (9.25,-1.86) and (4.4,-0.4) .. (0,0) .. controls (4.4,0.4) and (9.25,1.86) .. (14.54,4.38)   ;
\draw [color={rgb, 255:red, 208; green, 2; blue, 27 }  ,draw opacity=1 ][line width=3]    (245.1,116.14) -- (286.97,64.04) ;
\draw [shift={(290.1,60.14)}, rotate = 128.78] [color={rgb, 255:red, 208; green, 2; blue, 27 }  ,draw opacity=1 ][line width=3]    (14.54,-4.38) .. controls (9.25,-1.86) and (4.4,-0.4) .. (0,0) .. controls (4.4,0.4) and (9.25,1.86) .. (14.54,4.38)   ;
\draw [color={rgb, 255:red, 74; green, 144; blue, 226 }  ,draw opacity=1 ][line width=3]    (294.5,65.5) -- (294.85,125.93) ;
\draw [shift={(294.88,130.93)}, rotate = 269.67] [color={rgb, 255:red, 74; green, 144; blue, 226 }  ,draw opacity=1 ][line width=3]    (14.54,-4.38) .. controls (9.25,-1.86) and (4.4,-0.4) .. (0,0) .. controls (4.4,0.4) and (9.25,1.86) .. (14.54,4.38)   ;
\draw [color={rgb, 255:red, 208; green, 2; blue, 27 }  ,draw opacity=1 ][line width=3]    (299.1,129.14) -- (340.97,77.04) ;
\draw [shift={(344.1,73.14)}, rotate = 128.78] [color={rgb, 255:red, 208; green, 2; blue, 27 }  ,draw opacity=1 ][line width=3]    (14.54,-4.38) .. controls (9.25,-1.86) and (4.4,-0.4) .. (0,0) .. controls (4.4,0.4) and (9.25,1.86) .. (14.54,4.38)   ;
\draw [color={rgb, 255:red, 74; green, 144; blue, 226 }  ,draw opacity=1 ][line width=3]    (348.5,78.5) -- (348.85,138.93) ;
\draw [shift={(348.88,143.93)}, rotate = 269.67] [color={rgb, 255:red, 74; green, 144; blue, 226 }  ,draw opacity=1 ][line width=3]    (14.54,-4.38) .. controls (9.25,-1.86) and (4.4,-0.4) .. (0,0) .. controls (4.4,0.4) and (9.25,1.86) .. (14.54,4.38)   ;
\draw [color={rgb, 255:red, 208; green, 2; blue, 27 }  ,draw opacity=1 ][line width=3]    (353.1,144.14) -- (394.97,92.04) ;
\draw [shift={(398.1,88.14)}, rotate = 128.78] [color={rgb, 255:red, 208; green, 2; blue, 27 }  ,draw opacity=1 ][line width=3]    (14.54,-4.38) .. controls (9.25,-1.86) and (4.4,-0.4) .. (0,0) .. controls (4.4,0.4) and (9.25,1.86) .. (14.54,4.38)   ;
\draw [color={rgb, 255:red, 74; green, 144; blue, 226 }  ,draw opacity=1 ][line width=3]    (402.5,93.5) -- (401.93,144.93) ;
\draw [shift={(401.88,149.93)}, rotate = 270.63] [color={rgb, 255:red, 74; green, 144; blue, 226 }  ,draw opacity=1 ][line width=3]    (14.54,-4.38) .. controls (9.25,-1.86) and (4.4,-0.4) .. (0,0) .. controls (4.4,0.4) and (9.25,1.86) .. (14.54,4.38)   ;
\draw [color={rgb, 255:red, 208; green, 2; blue, 27 }  ,draw opacity=1 ][line width=3]    (406.1,152.14) -- (447.97,100.04) ;
\draw [shift={(451.1,96.14)}, rotate = 128.78] [color={rgb, 255:red, 208; green, 2; blue, 27 }  ,draw opacity=1 ][line width=3]    (14.54,-4.38) .. controls (9.25,-1.86) and (4.4,-0.4) .. (0,0) .. controls (4.4,0.4) and (9.25,1.86) .. (14.54,4.38)   ;
\draw [color={rgb, 255:red, 144; green, 19; blue, 254 }  ,draw opacity=1 ][line width=3]    (460.5,100.5) -- (459.91,184.93) ;
\draw [shift={(459.88,189.93)}, rotate = 270.4] [color={rgb, 255:red, 144; green, 19; blue, 254 }  ,draw opacity=1 ][line width=3]    (14.54,-4.38) .. controls (9.25,-1.86) and (4.4,-0.4) .. (0,0) .. controls (4.4,0.4) and (9.25,1.86) .. (14.54,4.38)   ;
\draw [color={rgb, 255:red, 74; green, 144; blue, 226 }  ,draw opacity=1 ][line width=3]    (452.5,100.5) -- (451.91,184.93) ;
\draw [shift={(451.88,189.93)}, rotate = 270.4] [color={rgb, 255:red, 74; green, 144; blue, 226 }  ,draw opacity=1 ][line width=3]    (14.54,-4.38) .. controls (9.25,-1.86) and (4.4,-0.4) .. (0,0) .. controls (4.4,0.4) and (9.25,1.86) .. (14.54,4.38)   ;
\draw  [color={rgb, 255:red, 184; green, 233; blue, 134 }  ,draw opacity=1 ][fill={rgb, 255:red, 126; green, 211; blue, 33 }  ,fill opacity=1 ] (452.1,95.14) .. controls (452.1,92.1) and (454.56,89.64) .. (457.6,89.64) .. controls (460.64,89.64) and (463.1,92.1) .. (463.1,95.14) .. controls (463.1,98.18) and (460.64,100.64) .. (457.6,100.64) .. controls (454.56,100.64) and (452.1,98.18) .. (452.1,95.14) -- cycle ;
\draw    (338,174) -- (366,174) ;
\draw [shift={(368,174)}, rotate = 180] [color={rgb, 255:red, 0; green, 0; blue, 0 }  ][line width=0.75]    (10.93,-3.29) .. controls (6.95,-1.4) and (3.31,-0.3) .. (0,0) .. controls (3.31,0.3) and (6.95,1.4) .. (10.93,3.29)   ;
\draw    (514,82) -- (498.01,109.27) ;
\draw [shift={(497,111)}, rotate = 300.38] [color={rgb, 255:red, 0; green, 0; blue, 0 }  ][line width=0.75]    (10.93,-3.29) .. controls (6.95,-1.4) and (3.31,-0.3) .. (0,0) .. controls (3.31,0.3) and (6.95,1.4) .. (10.93,3.29)   ;
\draw    (363,42) -- (437.09,64.42) ;
\draw [shift={(439,65)}, rotate = 196.84] [color={rgb, 255:red, 0; green, 0; blue, 0 }  ][line width=0.75]    (10.93,-3.29) .. controls (6.95,-1.4) and (3.31,-0.3) .. (0,0) .. controls (3.31,0.3) and (6.95,1.4) .. (10.93,3.29)   ;

\draw (341,40) node  [font=\large] [align=left] {\begin{minipage}[lt]{68pt}\setlength\topsep{0pt}
IP-EOT
\end{minipage}};
\draw (430,174) node   [align=left] {\begin{minipage}[lt]{68pt}\setlength\topsep{0pt}
Sinkhorn
\end{minipage}};
\draw (175.44,227) node  [font=\large] [align=left] {\begin{minipage}[lt]{139.32pt}\setlength\topsep{0pt}
Feasible domain
\end{minipage}};
\draw (572.38,200.07) node  [font=\large] [align=left] {\begin{minipage}[lt]{101.16pt}\setlength\topsep{0pt}
Optimal solutions
\end{minipage}};
\draw (530,125) node   [align=left] {\begin{minipage}[lt]{68pt}\setlength\topsep{0pt}
Newton
\end{minipage}};
\draw (443,211.4) node [anchor=north west][inner sep=0.75pt]  [font=\large]  {$S_{\U_{\mathbf{q}}^{\mathbf{p}}}^{\beta }$};
\draw (451,53.4) node [anchor=north west][inner sep=0.75pt]  [font=\large]  {$\hat{P}^{( t)}$};
\draw (300,175.46) node   [align=left] {\begin{minipage}[lt]{46.92pt}\setlength\topsep{0pt}
Middle $\displaystyle \beta $
\end{minipage}};
\draw (521.94,74.46) node   [align=left] {\begin{minipage}[lt]{40.88pt}\setlength\topsep{0pt}
Large $\displaystyle \beta $
\end{minipage}};
\end{tikzpicture}
\caption{Schematic of the PSN scheduling. IP‑EOT constructs a warm‑start satisfying the same-kernel condition; for larger $\beta$, Newton method is used to reduce marginal residuals, while for smaller $\beta$, Sinkhorn iteration is employed.}
\label{fig:psn-architecture}
\end{figure}
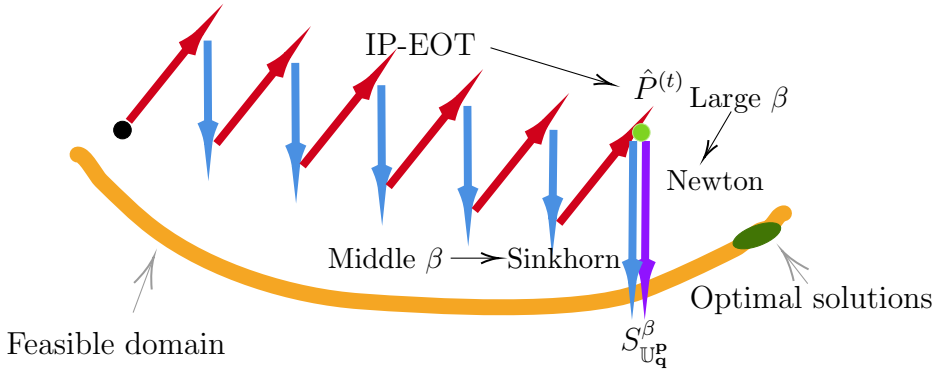

The Newton refinement solves the \emph{dual marginal equations}. For dual potentials
$\mathbf{x}\in\mathbb{R}^m$ and $\mathbf{y}\in\mathbb{R}^n$, define
\begin{equation}
P(\mathbf{x},\mathbf{y})
=
\expo\!\left(
\beta\left(
S+\mathbf{x}\mathbf{1}_n^\top
+\mathbf{1}_m\mathbf{y}^\top
\right)
\right),
\end{equation}
where the exponential is taken element-wise. The marginal residual is
\begin{equation}
R(\mathbf{x},\mathbf{y})
=
\begin{pmatrix}
P(\mathbf{x},\mathbf{y})\mathbf{1}_n-\mathbf{p}\\[2mm]
P(\mathbf{x},\mathbf{y})^\top\mathbf{1}_m-\mathbf{q}
\end{pmatrix}.
\end{equation}

Let
\[
\mathbf{z}
=
\begin{pmatrix}
\mathbf{x}\\
\mathbf{y}
\end{pmatrix},
\qquad
P_k=P(\mathbf{x}_k,\mathbf{y}_k).
\]
The Jacobian of the residual has the block structure
\begin{equation}
H_k
=
\beta
\begin{pmatrix}
\diag(P_k\mathbf{1}_n) & P_k\\
P_k^\top & \diag(P_k^\top\mathbf{1}_m)
\end{pmatrix}.
\end{equation}

Accordingly, the Newton direction is obtained from
\begin{equation}
H_k\Delta\mathbf{z}_k
=
-R(\mathbf{x}_k,\mathbf{y}_k),
\label{equ:newton-system}
\end{equation}
followed by the update
\begin{equation}
\mathbf{z}_{k+1}
=
\mathbf{z}_k+\Delta\mathbf{z}_k.
\label{equ:newton}
\end{equation}

The dual potentials are invariant under the transformation
\[
(\mathbf{x},\mathbf{y})
\mapsto
(\mathbf{x}+c\mathbf{1}_m,\,
\mathbf{y}-c\mathbf{1}_n),
\qquad c\in\mathbb{R},
\]
so the full Hessian $H_k$ has a one-dimensional null space. In practice,
this gauge freedom is removed, for example, by fixing one component of
$\mathbf{y}$, and the Newton direction is computed from the resulting
reduced nonsingular linear system.

To reduce the cost of Hessian solves, PSN employs {sparsification}: thresholding small entries of ${P}_k$ and solving the reduced system with incomplete Cholesky preconditioned conjugate gradient (PCG).

\begin{theorem}[Local Newton refinement \cite{Shen2026OpenReview}]
\label{thm:newton}
Under standard nondegeneracy assumptions and with initialization sufficiently close to the optimal scaling potentials, the exact Newton iteration achieves local quadratic convergence. With the sparse Hessian approximation, convergence is preserved provided the truncation error is controlled by the first-order residual and the approximated Hessian remains positive definite with bounded inverse.
\end{theorem}

A {Sinkhorn fallback} mechanism ensures robustness: if the Newton step fails to provide sufficient descent, the algorithm reverts to several Sinkhorn iterations before retrying. This hybrid strategy guarantees global convergence while exploiting local second-order acceleration when conditions are favorable. Figure \ref{fig.psn1} shows the performance of IP-EOT and PSN.

\section{A Comparative Synthesis}

Let us step back and compare the two regularization geometries across the two application domains.

\begin{table}[htbp]
\centering
\caption{Comparative synthesis of regularization geometries across domains.}
\label{tab:synthesis}

\renewcommand{\arraystretch}{1.2}
\begin{tabularx}{\linewidth}{
    @{}
    >{\raggedright\arraybackslash}p{0.22\linewidth}
    >{\raggedright\arraybackslash}X
    >{\raggedright\arraybackslash}X
    @{}
}
\toprule
\textbf{Dimension}
& \textbf{Graph Matching}
& \textbf{Optimal Transport} \\
\midrule

Coupling role
& Search direction in projected-gradient
& Final transport plan \\

Precision requirement
& Moderate (outer loop corrects)
& High (tight marginal feasibility) \\

Entropy/KL route
& CSGO, ASM: probabilistic directions via softassign
& IP-EOT warm-start + Sinkhorn/Newton for exact KL projection \\

Frobenius route
& FRAM: Euclidean directions via SDSN
& Less common \\

Key algorithmic tool
& Sinkhorn scaling (entropy); SDSN (Frobenius)
& Sinkhorn scaling + Newton refinement (entropy) \\

Parameter control
& $\beta=\gamma\log n$ (CSGO); adaptive $\beta$ (ASM)
& $\Delta\beta$ schedule (IP-EOT) \\

\bottomrule
\end{tabularx}
\end{table}

The unifying perspective reveals that {the same mathematical object}---a regularized coupling map on a matrix polytope---serves different purposes depending on the outer problem. In graph matching, it is a {direction generator}; in optimal transport, it is a {solution representation}. The regularizer determines the geometry: entropy induces an information-geometric structure; Frobenius corresponds to Euclidean.

\section{Connections to Broader Mathematics}

\subsection{Information Geometry and Bregman Divergences}

The entropy path is deeply connected to {information geometry}. The KL projection $\Gamma^{\KL}_{\mathbb{U}}$ is a special case of {Bregman projection} \cite{Bregman1967, Csiszar1975}, where the convex function is the negative entropy $\mathcal{H}(X) = \sum_{i,j} X_{ij}\log X_{ij}$. The Bregman divergence:
\begin{equation}
D_\mathcal{H}(X \| Y) = \mathcal{H}(X) - \mathcal{H}(Y) - \langle \nabla\mathcal{H}(Y), X-Y \rangle
\end{equation}
generalizes squared Euclidean distance and induces a Riemannian metric on the simplex. Sinkhorn scaling can be interpreted as alternating KL/Bregman projections onto the affine sets enforcing the prescribed row and column marginals \cite{Benamou2015}. The negative-entropy geometry induces the Fisher information metric on the positive simplex. Separately, optimal transport provides the metric structure underlying variational formulations of Wasserstein gradient flows \cite{Jordan1998}.

The transportation polytope $\mathbb{U}^\mathbf{p}_\mathbf{q}$ is closely related to the {Coupling Matrix Manifold} (CMM) $\mathbb{C}_n^m(\mathbf{p}, \mathbf{q})$ introduced by Shi et al.~\cite{Shi2021}, which is defined as the set of \emph{strictly positive} coupling matrices with the same marginal constraints:
\begin{equation}
\mathbb{C}_n^m(\mathbf{p}, \mathbf{q}) = \left\{ X \in \R_{++}^{m \times n} : X\mathbf{1}_n = \mathbf{p},\, X^\top\mathbf{1}_m = \mathbf{q} \right\}.
\end{equation}
The CMM is a smooth Riemannian manifold of dimension $(m-1)(n-1)$, equipped with the Fisher information metric \cite{Shi2021}. In contrast, our matrix polytope $\mathbb{U}^\mathbf{p}_\mathbf{q}$ includes the boundary (nonnegative entries) and is a convex polytope rather than a manifold. When $\mathbf{p}\in\R_{++}^m$ and $\mathbf{q}\in\R_{++}^n$, the relative interior of $\mathbb{U}^\mathbf{p}_\mathbf{q}$ coincides with $\mathbb{C}_n^m(\mathbf{p},\mathbf{q})$. This distinction is important: the CMM framework enables Riemannian optimization algorithms (gradient descent, trust region), while our matrix polytope framework enables convex analysis and regularized projection methods.

\subsection{Gromov--Wasserstein and Second-Order Matching}

Graph matching can be viewed through the lens of {Gromov--Wasserstein} (GW) distance \cite{Gromov1999, Memoli2011, Peyre2016GW}, which compares metric measure spaces by optimizing over couplings that preserve pairwise distances. Under a squared-loss GW formulation and fixed marginals, the coupling-dependent quadratic term can be reduced, up to marginal-dependent constants and sign conventions, to a form closely related to the Koopmans--Beckmann quadratic objective. The {Fused Gromov--Wasserstein} (FGW) framework \cite{Vayer2019} extends this to simultaneously account for node features (first-order) and structural similarity (second-order), providing a unified metric for attributed graph matching.

\subsection{Mixed-Precision Computing and Numerical Analysis}

The error-damping analysis of SDSN connects to the broader theme of {mixed-precision algorithms} in numerical linear algebra \cite{Carson2018, Higham2022}. Just as iterative refinement can recover full-precision solutions from low-precision LU factors, SDSN's alternating projection structure damps truncation errors through the zero-sum subspace projection. This observation opens the door to hardware-aware algorithm design, where GPU tensor cores (optimized for FP16) can accelerate the inner iterations while CPU FP64 handles the outer loop.

\section{Concluding Remarks}

This exposition has traced a path from the classical Birkhoff--von Neumann theorem through two modern regularization geometries---entropy/KL and the Frobenius/Euclidean---to contemporary algorithms for graph matching and optimal transport. The unifying concept is the {regularized coupling map} on a matrix polytope: a single mathematical object that changes its role from direction generator to solution representation depending on the problem context.

Several directions merit further exploration:

\begin{enumerate}[leftmargin=2em]
\item \textbf{Unified convergence theory}: Can the projected gradient outer-loop analysis for graph matching and the proximal-point theory for optimal transport be subsumed under a single variational framework?

\item \textbf{Infinite-dimensional limits}: As $n \to \infty$, do the discrete regularization paths converge to continuum objects? The entropic path connects to the Schr\'{o}dinger bridge problem \cite{Leonard2014}; the Frobenius path may relate to gradient flows in the Wasserstein space with $L^2$ regularization.

\item \textbf{Randomized and sketching variants}: For truly massive problems ($n \sim 10^6$), can randomized sampling or matrix sketching reduce the effective dimension of the coupling while preserving the regularization structure?

\item \textbf{Quantum analogues}: The Birkhoff--von Neumann theorem has quantum extensions to doubly stochastic operators and unital channels. Do regularized coupling maps have meaningful quantum counterparts, perhaps via quantum Sinkhorn scaling \cite{Gurvits2002Quantum}?

\item \textbf{Riemannian-regularized hybrid methods}: Can the Riemannian geometry of the coupling matrix manifold \cite{Shi2021} be combined with our regularized coupling map framework to develop hybrid algorithms that exploit both the geometric structure (via Riemannian optimization) and the regularization structure (via our coupling maps)?
\end{enumerate}

The matrix polytope, humble as it may seem, continues to reveal deep connections between combinatorics, optimization, geometry, and computation. We hope this exposition encourages mathematicians across these fields to explore the rich structure of regularized coupling maps.

\section*{Acknowledgments}

The authors thank their collaborators Qiang Niu, Yuan Liang, and Chujie Ouyang for their essential contributions to the research surveyed here. This work was supported by the National Key R\&D Program of China (2025YFG0202100, 2025YFG0202600) and the Guangdong Provincial Key Laboratory of Interdisciplinary Research and Application for Data Science.


\begin{thebibliography}{99}

\bibitem{Beckmann1957}
T. C. Koopmans and M. J. Beckmann, \textit{Assignment Problems and the Location of Economic Activities}, Econometrica \textbf{25} (1957), 53--76.

\bibitem{Beck2003}
A. Beck and M. Teboulle, \textit{Mirror descent and nonlinear projected subgradient methods for convex optimization}, Operations Research Letters \textbf{31} (2003), 167--175.

\bibitem{Bauschke1996}
H. H. Bauschke and J. M. Borwein, \textit{On projection algorithms for solving convex feasibility problems}, SIAM Review \textbf{38} (1996), 367--426.

\bibitem{Benamou2015}
J.-D. Benamou, G. Carlier, M. Cuturi, L. Nenna, and G. Peyr\'{e}, \textit{Iterative Bregman projections for regularized transportation problems}, SIAM J. Sci. Comput. \textbf{37} (2015), A1111--A1138.

\bibitem{Birkhoff1946}
G. Birkhoff, \textit{Tres observaciones sobre el \'{a}lgebra lineal}, Univ. Nac. Tucum\'{a}n Rev. Ser. A \textbf{5} (1946), 147--151.

\bibitem{Bregman1967}
L. M. Bregman, \textit{The relaxation method of finding the common point of convex sets and its application to the solution of problems in convex programming}, USSR Comput. Math. Math. Phys. \textbf{7} (1967), 200--217.

\bibitem{Carson2018}
E. Carson and N. J. Higham, \textit{Accelerating the solution of linear systems by iterative refinement in three precisions}, SIAM J. Sci. Comput. \textbf{40} (2018), A817--A847.

\bibitem{garey1979computers} Michael R.~Garey and David S.~Johnson, \emph{Computers and Intractability: A Guide to the Theory of NP-Completeness}, W.H. Freeman \& Co., 1979.

\bibitem{Cominetti1994}
R. Cominetti and J. San Mart\'{i}n, \textit{Asymptotic analysis of the exponential penalty trajectory in linear programming}, Math. Program. \textbf{67} (1994), 169--187.

\bibitem{Cuturi2013}
M. Cuturi, \textit{Sinkhorn distances: Lightspeed computation of optimal transport}, in \textit{Advances in Neural Information Processing Systems}, 2013, pp. 2292--2300.

\bibitem{Csiszar1975}
I. Csisz\'{a}r, \textit{I-divergence geometry of probability distributions and minimization problems}, Ann. Probab. \textbf{3} (1975), 146--158.

\bibitem{Deming1940}
W. E. Deming and F. F. Stephan, \textit{On a least squares adjustment of a sampled frequency table when the expected marginal totals are known}, Ann. Math. Statist. \textbf{11} (1940), 427--444.

\bibitem{Douik2018}
A. Douik and B. Hassibi, \textit{Manifold optimization for optimal transport}, arXiv:1810.00212, 2018.

\bibitem{Frank1956}
M. Frank and P. Wolfe, \textit{An algorithm for quadratic programming}, Naval Res. Logist. Quart. \textbf{3} (1956), 95--110.

\bibitem{Gromov1999}
M. Gromov, \textit{Metric Structures for Riemannian and Non-Riemannian Spaces}, Birkh\"{a}user, 1999.

\bibitem{Gurvits2002}
L. Gurvits and A. Samorodnitsky, \textit{A deterministic algorithm for approximating the mixed discriminant and mixed volume, and a combinatorial corollary}, Discrete Comput. Geom. \textbf{27} (2002), 531--550.

\bibitem{Higham2022}
N. J. Higham and T. Mary, \textit{Mixed precision algorithms in numerical linear algebra}, Acta Numer. \textbf{31} (2022), 347--414.

\bibitem{Idel2016}
M. Idel, \textit{A review of matrix scaling and Sinkhorn's normal form for matrices and positive maps}, arXiv:1609.06349, 2016.

\bibitem{Jordan1998}
R. Jordan, D. Kinderlehrer, and F. Otto, \textit{The variational formulation of the Fokker-Planck equation}, SIAM J. Math. Anal. \textbf{29} (1998), 1--17.

\bibitem{Kruithof1937}
J. Kruithof, \textit{Telefoonverkeersrekening}, De Ingenieur \textbf{52} (1937), E15--E25.

\bibitem{Lawler1963}
E. L. Lawler, \textit{The quadratic assignment problem}, Management Sci. \textbf{9} (1963), 586--599.

\bibitem{Leonard2014}
C. L\'{e}onard, \textit{A survey of the Schr\'{o}dinger problem and some of its connections with optimal transport}, Discrete Contin. Dyn. Syst. A \textbf{34} (2014), 1533--1574.

\bibitem{Memoli2011}
F. M\'{e}moli, \textit{Gromov--Wasserstein distances and the metric approach to object matching}, Found. Comput. Math. \textbf{11} (2011), 417--487.

\bibitem{Peyre2016GW}
G. Peyr\'{e}, M. Cuturi, and J. Solomon, \textit{Gromov-Wasserstein averaging of kernel and distance matrices}, in \textit{Proc. ICML}, 2016, pp. 2664--2672.

\bibitem{Peyre2019}
G. Peyr\'{e} and M. Cuturi, \textit{Computational optimal transport}, Found. Trends Mach. Learn. \textbf{11} (2019), 355--607.

\bibitem{Rockafellar1970}
R. T. Rockafellar, \textit{Convex Analysis}, Princeton University Press, 1970.

\bibitem{Sinkhorn1964}
R. Sinkhorn, \textit{A relationship between arbitrary positive matrices and doubly stochastic matrices}, Ann. Math. Statist. \textbf{35} (1964), 876--879.

\bibitem{SinkhornKnopp1967}
R. Sinkhorn and P. Knopp, \textit{Concerning nonnegative matrices and doubly stochastic matrices}, Pacific J. Math. \textbf{21} (1967), 343--348.

\bibitem{SomaUschmajew2026}
T. Soma and A. Uschmajew, \textit{Accelerating operator Sinkhorn iteration with overrelaxation}, Math. Program. (2026), \url{https://doi.org/10.1007/s10107-026-02361-1}.

\bibitem{Shi2021}
Y. Shi, J. Gao, X. Hong, S. T. B. Choy, and Z. Wang, \textit{Coupling matrix manifolds assisted optimization for optimal transport problems}, Machine Learning \textbf{110} (2021), 533--558.

\bibitem{Shen2024CVPR}
B. Shen, Q. Niu, and S. Zhu, \textit{Adaptive softassign via Hadamard-equipped Sinkhorn}, in \textit{Proc. CVPR}, 2024.

\bibitem{Shen2026CSGO}
B. Shen, Q. Niu, and S. Zhu, \textit{CSGO: Constrained-softassign gradient optimization for large graph matching}, Pattern Recognition \textbf{177} (2026), 113329.

\bibitem{Shen2025NeurIPS}
B. Shen, Y. Liang, and S. Zhu, \textit{FRAM: Frobenius-regularized assignment matching with mixed-precision computing}, in \textit{Proc. NeurIPS}, 2025.

\bibitem{Shen2026OpenReview}
B. Shen, C. Ouyang, and S. Zhu, \textit{A Proximal-Sinkhorn-Newton method for entropic optimal transport}, OpenReview, 2026.

\bibitem{Lu2016}
Y. Lu, K. Huang, and C.-L. Liu,
A fast projected fixed-point algorithm for large graph matching,
\emph{Pattern Recognition},
vol. 60, pp. 971--982, 2016.

\bibitem{Gurvits2002Quantum}
L. Gurvits,
\textit{Quantum Matching Theory (with New Complexity-Theoretic,
Combinatorial and Topological Insights on the Nature of the Quantum Entanglement)},
arXiv:quant-ph/0201022, 2002.

\bibitem{Gold1996}
S. Gold and A. Rangarajan,
A graduated assignment algorithm for graph matching,
\emph{IEEE Transactions on Pattern Analysis and Machine Intelligence},
vol. 18, no. 4, pp. 377--388, 1996.


\bibitem{Vayer2019}
T. Vayer, L. Chapel, R. Flamary, R. Tavenard, and N. Courty, \textit{Optimal transport for structured data with application on graphs}, in \textit{Proc. ICML}, 2019, pp. 6275--6284.

\bibitem{vonNeumann1953}
J. von Neumann, \textit{A certain zero-sum two-person game equivalent to the optimal assignment problem}, in \textit{Contributions to the Theory of Games}, Ann. Math. Studies \textbf{28}, Princeton University Press, 1953, pp. 5--12.

\bibitem{xie2020fast}
Yujia Xie, Xiangfeng Wang, Ruijia Wang, and Hongyuan Zha.
\newblock A fast proximal point method for computing exact Wasserstein distance.
\newblock In \emph{Proceedings of the 35th Uncertainty in Artificial Intelligence Conference},
volume 115, pages 433--453. PMLR, 2020.


\bibitem{Villani2009}
C. Villani, \textit{Optimal Transport: Old and New}, Grundlehren der mathematischen Wissenschaften \textbf{338}, Springer, 2009.

\end{thebibliography}
\end{document}